\documentclass[a4paper,12pt]{amsart}
\usepackage[margin=3cm]{geometry}
\usepackage[applemac]{inputenc}
\usepackage{eurosym}

\newtheorem{theorem}{Theorem}[section]
\newtheorem{rem} [theorem] {Remark}
\newtheorem{prop} [theorem]{Proposition}

\newtheorem{definition}[theorem]{Definition}
\newcommand{\ovprt}{\overline{\partial}}
\newcommand{\ovli}{\overline}
\newcommand{\dquer}{\overline\partial}
\newcommand{\dquers}{\overline\partial ^*_\varphi}
\newcommand{\boxphi}{\square_\varphi}
\newcommand{\levim}{\frac{\partial^2\varphi}{\partial z_j\partial\overline z_k}}

\numberwithin{equation}{section}
\begin{document}
\title{Compactness  of  the  $\ovprt $ - Neumann operator on weighted $(0,q)$- forms.}

\author{ Friedrich Haslinger}


 \address{ F. Haslinger: Institut f\"ur Mathematik, Universit\"at Wien,
Nordbergstrasse 15, A-1090 Wien, Austria}
\email{ friedrich.haslinger@univie.ac.at}

\maketitle

\begin{abstract} ~\\
 As an application of a new characterization of   compactness of the $\ovprt $-Neumann operator we derive a sufficient condition  for compactness  of the $\ovprt $- Neumann operator  on $(0,q)$-forms in weighted $L^2$-spaces on $\mathbb{C}^n.$

\end{abstract}

\section{Introduction.}~\\

 In this paper we continue the investigations of \cite{HaHe} ans \cite{Has5} concerning  existence and compactness of the canonical solution operator to $\ovprt $ on weighted
$L^2$-spaces over $\mathbb C^n.$

Let $\varphi : \mathbb C^n \longrightarrow \mathbb R^+ $ be a plurisubharmonic $\mathcal C^2$-weight function and define the space
$$L^2(\mathbb C^n , \varphi )=\{ f:\mathbb C^n \longrightarrow \mathbb C \ : \ \int_{\mathbb C^n}
|f|^2\, e^{-\varphi}\,d\lambda < \infty \},$$
where $\lambda$ denotes the Lebesgue measure, the space $L^2_{(0,q)}(\mathbb C^n, \varphi )$ of $(0,q)$-forms with coefficients in
$L^2(\mathbb C^n , \varphi ),$ for $1\le q \le n.$ 
Let 
$$( f,g)_\varphi=\int_{\mathbb{C}^n}f \,\overline{g} e^{-\varphi}\,d\lambda$$
denote the inner product and 
$$\| f\|^2_\varphi =\int_{\mathbb{C}^n}|f|^2e^{-\varphi}\,d\lambda $$
the norm in $L^2(\mathbb C^n , \varphi ).$

We consider the weighted
$\ovprt $-complex 
$$
L^2_{(0,q-1)}(\mathbb C^n , \varphi )\underset{\underset{\ovprt_\varphi^* }
\longleftarrow}{\overset{\ovprt }
{\longrightarrow}} L^2_{(0,q)}(\mathbb C^n , \varphi )
\underset{\underset{\ovprt_\varphi^* }
\longleftarrow}{\overset{\ovprt }
{\longrightarrow}} L^2_{(0,q+1)}(\mathbb C^n , \varphi ),
$$
where 
 for $(0,q)$-forms $u=\sum_{|J|=q}' u_J\,d\overline z_J$ with coefficients in $\mathcal{C}_0^\infty (\mathbb{C}^n)$ we have
$$\ovprt u = \sum_{|J|=q}\kern-1pt{}^{\prime}  \, \sum_{j=1}^n \frac{\partial u_J}{\partial \ovli z_j}\, d\ovli z_j \wedge d \ovli z_J,$$
and 
$$ \ovprt_\varphi ^* u = -  \sum_{|K|=q-1}\kern-7pt{}^{\prime}  \, \sum_{k=1}^n \delta_k u_{kK}\,d\ovli z_K,$$
where $\delta_k=\frac{\partial}{\partial z_k}-\frac{\partial \varphi}{\partial z_k}.$

There is an interesting  connection between $\ovprt $ and the theory of Schr\"odinger operators\index{Schr\"odinger operator} with magnetic fields,
see for example \cite{Ch}, \cite{B}, \cite{FS3} and \cite{ChF} for recent contributions exploiting this point of view.

The complex Laplacian on $(0,q)$-forms is defined as
$$\boxphi := \dquer  \,\dquers + \dquers \dquer,$$
where the symbol $\boxphi $ is to be understood as the maximal closure of the operator initially defined on forms with coefficients in $\mathcal{C}_0^\infty$, i.e., the space of smooth functions with compact support.

$\boxphi $ is a selfadjoint and positive operator, which means that 
$$( \boxphi f,f)_\varphi \ge 0 \ , \   {\text{for}} \  f\in dom (\boxphi ).$$
The associated Dirichlet form is denoted by 
\begin{equation}\label{diri}
Q_\varphi (f,g)= ( \dquer f,\dquer g)_\varphi + ( \dquers f ,\dquers g)_\varphi, 
\end{equation}
for $f,g\in dom (\dquer ) \cap dom (\dquers ).$ The weighted $\dquer $-Neumann operator
$N_{\varphi, q} $ is - if it exists - the bounded inverse of $\boxphi .$ 
\vskip 0.5 cm
We indicate that a $(0,1)$-form $f=\sum_{j=1}^n f_j\,d\ovli z_j$ belongs to  $dom(\dquers)$ if and only if 
\begin{equation*}
 \sum_{j=1}^n\left ( \frac{\partial f_j}{\partial z_j}- \frac{\partial \varphi}{\partial z_j}\, f_j \right )
 \in L^2(\mathbb{C}^n, \varphi )
 \end{equation*}
  and that forms with coefficients in $\mathcal{C}_0^\infty(\mathbb{C}^n)$ are dense in $dom(\dquer)\cap dom(\dquers)$ in the graph norm $f\mapsto (\Vert \dquer f\Vert _\varphi^2+\Vert \dquers f\Vert _\varphi^2)^\frac{1}{2}$ (see \cite{GaHa}).
\vskip 0.5 cm
We consider  the Levi - matrix 
$$M_\varphi=\left(\levim\right)_{jk}$$
of $\varphi $ and suppose that
 the sum $s_q$ of any $q$ (equivalently: the smallest $q$) eigenvalues of $M_\varphi$ satisfies
\begin{equation}\label{sq}
\liminf_{|z|\to \infty}s_q(z) > 0.
\end{equation}
We show that (\ref{sq}) implies that there exists a continuous linear operator
$$N_{\varphi, q} : L^2_{(0,q)}( \mathbb{C}^n, \varphi) \longrightarrow L^2_{(0,q)}( \mathbb{C}^n, \varphi),$$
such that $\boxphi \circ N_{\varphi,q} u = u,$ for any $u\in L^2_{(0,q)}( \mathbb{C}^n, \varphi).$
\vskip 0.3 cm
If we suppose that that the sum $s_q$ of any $q$ (equivalently: the smallest $q$) eigenvalues of $M_\varphi$ satisfies
\begin{equation}\label{sqcomp}
\lim_{|z|\to \infty}s_q(z) =\infty .
\end{equation}
Then the $\ovprt$-Neumann operator $N_{\varphi, q} : L^2_{(0,q)}( \mathbb{C}^n, \varphi) \longrightarrow L^2_{(0,q)}( \mathbb{C}^n, \varphi)$ is compact.

This generalizes results from \cite{HaHe} and \cite{Has5}, where the case of  $q=1$ was handled.

Finally we discuss some examples in $\mathbb{C}^2.$

\vskip 1 cm

\section{The weighted Kohn-Morrey formula}~\\

\vskip 0.5 cm
First we compute 
$$( \boxphi u,u)_\varphi =  \| \ovprt u \|^2_\varphi + \| \ovprt_\varphi ^* u\|^2_\varphi $$
for $u\in dom (\boxphi ).$

We obtain 
\begin{align*}\| \ovprt u \|^2_\varphi + \| \ovprt_\varphi ^* u\|^2_\varphi &= \sum_{|J|=|M|=q}\kern-10pt{}^{\prime} \kern8pt   \sum_{j,k=1}^n\, \epsilon_{jJ}^{kM}\, \int_{\mathbb{C}^n}\frac{\partial u_J}{\partial \ovli z_j}
\ovli{\frac{\partial u_M}{\partial \ovli z_k}}\, e^{-\varphi}\,d\lambda \\
&+  \sum_{|K|=q-1}\kern-7pt{}^{\prime}  \, \sum_{j,k=1}^n\int_{\mathbb{C}^n}\, \delta_j u_{jK} \ovli{\delta_k u_{kK}}\, e^{-\varphi}\,d\lambda,
\end{align*}
where $ \epsilon_{jJ}^{kM}=0$ if $j\in J$ or $k\in M$ or if ${k} \cup M \neq {j} \cup J,$ and equals the sign of the permutation $\binom{kM}{jJ}$ otherwise.
The right-hand side  of the last formula can be rewritten as 
$$\sum_{|J|=q}\,^{'} \, \sum_{j=1}^n\, \left \| \frac{\partial u_J}{\partial \ovli z_j}\right \|_\varphi ^2 +
 \sum_{|K|=q-1}\kern-10pt{}^{\prime}  \, \sum_{j,k=1}^n \int_{\mathbb{C}^n}\left (  \delta_j u_{jK} \ovli{\delta_k u_{kK}} - 
 \frac{\partial u_{jK}}{\partial \ovli z_k}
\ovli{\frac{\partial u_{kK}}{\partial \ovli z_j}} \right )\, e^{-\varphi}\,d\lambda,$$
see \cite{Str} Proposition 2.4 for the details. Now we mention that 
 for $f,g\in \mathcal{C}^\infty_0(\mathbb{C}^n)$ we have
$$\left ( \frac{\partial f}{\partial \overline{z}_k}, g\right )_\varphi =
- ( f, \delta_kg )_\varphi
$$
and hence
$$\left ( \left [ \delta_j, \frac{\partial}{\partial \overline z_k} \right ]u_{jK}, u_{kK} \right )_\varphi =- \left ( \frac{\partial u_{jK}}{\partial \overline z_k}, \frac{\partial u_{kK}}{\partial \overline z_j} \right )_\varphi + ( \delta_j u_{jK}, \delta_k u_{kK} )_\varphi.$$
Since
$$ \left [ \delta_j,\frac{\partial}{\partial\overline{z}_k} \right ]= 
\frac{\partial^2 \varphi}{\partial z_j \partial\overline{z}_k},$$

 we get
\begin{equation}\label{komo6}
\| \ovprt u \|^2_\varphi + \| \ovprt_\varphi ^* u\|^2_\varphi = 
\sum_{|J|=q}\,^{'} \, \sum_{j=1}^n\, \left \| \frac{\partial u_J}{\partial \ovli z_j}\right \|_\varphi ^2 +
 \sum_{|K|=q-1}\kern-7pt{}^{\prime}   \, \, \sum_{j,k=1}^n \int_{\mathbb{C}^n} 
 \frac{\partial^2 \varphi}{\partial z_j \partial\overline{z}_k}\,u_{jK}\overline{u}_{kK}\, e^{-\varphi}\,d\lambda.
\end{equation}
Formula (\ref{komo6}) is a version of the Kohn -Morrey formula, compare \cite{Str} or \cite{McN2}.
\vskip 0.3 cm
\begin{prop}\label{weiq}
Let $1\le q \le n$ and suppose that the sum $s_q$ of any $q$ (equivalently: the smallest $q$) eigenvalues of $M_\varphi$ satisfies
\begin{equation}\label{sq}
\liminf_{|z|\to \infty}s_q(z) > 0.
\end{equation}
Then there exists a uniquely determined bounded linear operator
$$N_{\varphi, q} : L^2_{(0,q)}( \mathbb{C}^n, \varphi) \longrightarrow L^2_{(0,q)}( \mathbb{C}^n, \varphi),$$
such that $\boxphi \circ N_{\varphi,q} u = u,$ for any $u\in L^2_{(0,q)}( \mathbb{C}^n, \varphi).$
\end{prop}
\begin{proof} Let $\mu_{\varphi ,1}\le   \mu_{\varphi ,2} \le \dots \le \mu_{\varphi , n}$ denote the eigenvalues of $M_\varphi$ and suppose that $M_\varphi$ is diagonalized. Then, in a suitable basis, 
\begin{align*}
 \sum_{|K|=q-1}\kern-7pt{}^{\prime}  \, \sum_{j,k=1}^n  
 \frac{\partial^2 \varphi}{\partial z_j \partial\overline{z}_k}\,u_{jK}\overline{u}_{kK}& =  \sum_{|K|=q-1}\kern-7pt{}^{\prime} \,  \, \sum_{j=1}^n \mu_{\varphi , j} |u_{jK}|^2\\
 &= \sum_{J=(j_1, \dots , j_q)}\kern-10pt{}^{\prime} \,  (\mu_{\varphi , j_1} + \dots + \mu_{\varphi , j_q} )|u_J|^2 \\
 &\ge s_q |u|^2
\end{align*}
It follows from (\ref{komo6}) that there exists a constant $C>0$ such that
\begin{equation}\label{weigh6}
\|u\|^2_\varphi \le C( \| \ovprt u \|^2_\varphi + \| \ovprt_\varphi ^* u\|^2_\varphi )
\end{equation}
for each $(0,q)$-form $u\in $dom\,$(\ovprt ) \,\cap$
dom\,$(\ovprt_\varphi^*).$
For a given $v \in  L^2_{(0,q)}( \mathbb{C}^n, \varphi)$ consider the linear functional $L$ on  dom\,($\ovprt) \,\cap$ dom\,($\ovprt_\varphi^*)$ given by $L(u)=(u,v)_\varphi .$ Notice that   dom\,($\ovprt) \,\cap$ dom\,($\ovprt_\varphi^*)$ is a Hilbertspace in the inner product $Q_\varphi .$ 
Since we have by  \ref{weigh6} 
$$|L(u)| = |(u,v)_\varphi | \le  \|u\|_\varphi \,  \|v\|_\varphi \le C Q_\varphi (u,u)^{1/2} \,  \|v\|_\varphi .$$
Hence by the Riesz reprentation theorem there exists a uniquely determined $(0,q)$-form $N_{\varphi, q} v $ such that
$$(u,v)_\varphi = Q_\varphi (u, N_{\varphi, q} v )= (\ovprt u, \ovprt N_{\varphi,q} v)_\varphi + (\ovprt_\varphi^* u, \ovprt_\varphi^* N_{\varphi, q} v )_\varphi,$$
from which we immediately get that $\boxphi \circ N_{\varphi,q} v = v,$ for any $v\in L^2_{(0,q)}( \mathbb{C}^n, \varphi).$
If we set $u=N_{\varphi, q} v$ we get again from \ref{weigh6}
$$\|\ovprt N_{\varphi, q} v \|_\varphi ^2 + \|\ovprt_\varphi^* N_{\varphi, q} v \|_\varphi ^2 = Q_\varphi (N_{\varphi, q} v , N_{\varphi, q} v )=(N_{\varphi, q} v , v)_\varphi \le \| N_{\varphi, q} v \|_\varphi \, \| v \|_\varphi $$
$$\le C_1 ( \|\ovprt N_{\varphi, q} v \|_\varphi ^2 + \|\ovprt_\varphi^* N_{\varphi, q} v \|_\varphi ^2 )^{1/2} \ \|v\|_\varphi ,$$
hence 
$$( \|\ovprt N_{\varphi, q} v \|_\varphi ^2 + \|\ovprt_\varphi^* N_{\varphi, q} v \|_\varphi ^2 )^{1/2} \le C_2 \|v\|_\varphi $$
and finally again by \ref{weigh6}
$$\|N_{\varphi, q} v\|_\varphi  \le C_3 (\|\ovprt N_{\varphi, q} v \|_\varphi ^2 + \|\ovprt_\varphi^* N_{\varphi, q} v \|_\varphi ^2)^{1/2} \le  C_4 \|v\|_\varphi , $$
where $C_1, C_2, C_3, C_4>0$ are constants. Hence we get that $N_{\varphi, q}$ is a continuous linear operator from $L^2_{(0,q)}( \mathbb{C}^n, \varphi)$ into itself (see also \cite{H} or \cite{ChSh}). 
\end{proof}
\vskip 0.5 cm

\section{Compactness of $N_{\varphi, q}$}~\\

We use a characterization of precompact subsets of $L^2$-spaces, see \cite{AF}:

A bounded subset $\mathcal{A}$ of $L^2(\Omega )$ is precompact\index{precompact} in  $L^2(\Omega )$ if and only if for every $\epsilon >0$ there exists a number $\delta >0$ and a subset $\omega \subset\subset \Omega$ such that for every $u\in  \mathcal{A}$ and $h\in \mathbb{R}^n$ with $|h|<\delta$ both of the following inequalities hold:
\begin{equation}\label{com1}
 (i) \  \int_\Omega | \tilde u (x+h) - \tilde u (x)|^2\, dx < \epsilon^2  \hskip 0.3 cm , \hskip 0.3 cm (ii) \  \int_{\Omega \setminus \overline\omega} |u(x)|^2 \,dx < \epsilon^2.
 \end{equation}

\vskip 0.3 cm

In addition  we define an appropriate Sobolev space and prove compactness of the corresponding embedding, for related settings see \cite{BDH}, \cite{Jo}, \cite{KM} .

\begin{definition} Let
$$\mathcal{W}_q^{Q_\varphi}= \{ u\in L^2_{(0,q)}(\mathbb{C}^n , \varphi) \ : \  \| \ovprt u \|^2_\varphi + \| \ovprt_\varphi ^* u\|^2_\varphi< \infty \}$$ with norm
$$\| u\|_{Q_\varphi}  =  (\| \ovprt u \|^2_\varphi + \| \ovprt_\varphi ^* u\|^2_\varphi )^{1/2}.$$
\end{definition}

{\bf Remark:} $\mathcal{W}_q^{Q_\varphi}$ coincides with the form domain $dom (\dquer ) \cap dom (\dquers )$ of $Q_\varphi $ (see \cite{Ga}, \cite{GaHa} ).

\begin{prop}\label{cp q} 
Let $\varphi$ be a plurisubharmonic\index{plurisubharmonic} $\mathcal C^2$- weight function.
Let $1\le q \le n$ and suppose that the sum $s_q$ of any $q$ (equivalently: the smallest $q$) eigenvalues of $M_\varphi$ satisfies
\begin{equation}\label{sqcomp}
\lim_{|z|\to \infty}s_q(z) =\infty .
\end{equation}
Then $N_{\varphi, q} : L^2_{(0,q)}( \mathbb{C}^n, \varphi) \longrightarrow L^2_{(0,q)}( \mathbb{C}^n, \varphi)$ is compact.
\end{prop}
\begin{proof}
For $(0,q)$ forms one has by (\ref{komo6}) and Proposition \ref{weiq} that
\begin{equation}\label{compactq}
 \| \ovprt u \|^2_\varphi + \| \ovprt_\varphi ^* u\|^2_\varphi \ge  \int_{\mathbb{C}^n}s_q (z) \, |u(z)|^2\,  e^{-\varphi(z)}\,d\lambda (z) .
\end{equation}
We indicate that the embedding
$$j_{\varphi, q} : \mathcal{W}_q^{Q_\varphi } \hookrightarrow  L^2_{(0,q)}(\mathbb{C}^n, \varphi )$$ is compact by showing that the unit ball of $ \mathcal{W}_q^{Q_\varphi }$ is a precompact subset of $ L^2_{(0,q)}(\mathbb{C}^n, \varphi ),$ which follows by the above mentioned characterization of precompact subsets in $L^2$- spaces with the help of 
G\aa rding's inequality to verify (\ref{com1}) (i)(see for instance \cite{F} or \cite{ChSh}) and  to verify (\ref{com1}) (ii) : we have 
$$\int_{\mathbb{C}^n \setminus \mathbb{B}_R} |u(z)|^2 e^{-\varphi (z)}\,d\lambda (z)\le
\int_{\mathbb{C}^n \setminus \mathbb{B}_R} \frac{s_q(z) |u_(z)|^2}{\inf \{s_q(z)  :  |z|\ge R\}}   e^{-\varphi(z)}  d\lambda (z),$$
which implies by 
(\ref{compactq}) that
$$\int_{\mathbb{C}^n \setminus \mathbb{B}_R} |u(z)|^2 e^{-\varphi (z)}\,d\lambda (z)\le
 \frac{ \|u\|^2_{Q_\varphi}}{\inf \{s_q(z) \, : \, |z|\ge R\}} < \epsilon,$$
 if $R$ is big enough, 
 see \cite{Has5} for the details. 
 
 This together with the fact that 
 $N_{\varphi, q}= j_{\varphi, q} \circ j_{\varphi, q}^*,$ (see \cite{Str}) gives the desired result.

\end{proof}

\vskip 0.5 cm

\begin{rem}
If $q=1$ condition (\ref{sqcomp}) means that the lowest eigenvalue $\mu_{\varphi, 1}$ of $M_\varphi$ satisfies
\begin{equation}\label{com10}
\lim_{|z|\to \infty}\mu_{\varphi, 1} (z) =\infty .
\end{equation}
This implies compactness of $N_{\varphi ,1}$ (see \cite{Has5}).
\end{rem}
\vskip 0.5 cm

{\bf Examples:} a) We consider the plurisubharmonic weight function $\varphi (z,w)= |z|^2 |w|^2 + |w|^4$ on $\mathbb{C}^2.$ 
 The Levi matrix of $\varphi $ has the form 
$$\left(
\begin{array}{cc}
 |w|^2&\overline z w\\
 \overline w z&|z|^2 + 4|w|^2
 \end{array}\right) 
$$
and the eigenvalues are
$$\mu_{\varphi, 1} (z,w)=\frac{1}{2} \left ( 5|w|^2+|z|^2- \sqrt{9|w|^4+10|z|^2 |w|^2 + |z|^4} \right )$$
$$= \frac{16|w|^4}{2 \left ( 5|w|^2+|z|^2+ \sqrt{9|w|^4+10|z|^2 |w|^2 + |z|^4} \right )},$$
and 
$$\mu_{\varphi, 2} (z,w)=\frac{1}{2} \left ( 5|w|^2+|z|^2+ \sqrt{9|w|^4+10|z|^2 |w|^2 + |z|^4} \right ).$$
It follows that (\ref{com10}) fails, since even
$$\lim_{|z|\to \infty} |z|^2\mu_{\varphi, 1} (z,0)=0,$$
but 
$$s_2(z,w)=\frac{1}{4} \, \Delta \varphi (z,w) = |z|^2 + 5 |w|^2,$$
hence (\ref{sqcomp}) is satisfied for $q=2.$

\vskip 0.3 cm
b)In the next example we consider decoupled weights.
Let $n\ge 2$ and 
$$\varphi (z_1,z_2, \dots , z_n)=\varphi (z_1) + \varphi (z_2) + \dots + \varphi (z_n)$$
be a plurisubharmonic decoupled weight function and suppose that $|z_\ell |^2 \Delta \varphi_\ell (z_\ell )\to +\infty $, as $|z_\ell | \to \infty $ for some $\ell \in \{1, \dots , n \}.$ Then the $\ovprt $-Neumann operator$N_{\varphi, 1} $ acting on $L^2_{(0,1)}(\mathbb{C}^n, \varphi)$ fails to be compact (see \cite{HaHe}, \cite{Ga}, \cite{Schn06}).

Finally we discuss two examples in $\mathbb{C}^2:$ for $\varphi (z_1,z_2) = |z_1|^2+|z_2|^2$ all eigenvalues of the Levi matrix are $1$ and $N_{\varphi, 1}$ fails to be compact by the above result on decoupled weights, for the weightfunction $\varphi (z_1,z_2)=|z_1|^4 + |z_2|^4$ the eigenvalues are $4|z_1|^2$ and $4|z_2|^2$ and $N_{\varphi ,1}$ fails to be compact again by the above result, whereas $N_{\varphi ,2}$ is compact by \ref{cp q}.


\vskip 1 cm

\begin{thebibliography}{1}
\bibitem{AF} R.A. Adams and J.J.F. Fournier, \textit{Sobolev spaces.} Pure and Applied Math. Vol.~140, Academic Press, 2006.
\bibitem{B} B. Berndtsson, \textit{$\ovprt$ and Schr\"odinger operators.} Math. Z. \textbf{221} (1996), 401--413.
\bibitem{BDH} P. Bolley, M. Dauge and B. Helffer, \textit{Conditions suffisantes pour l'injection compacte d'espace de Sobolev \`a poids.} S\'eminaire \'equation aux d\'eriv\'ees partielles (France), Universit\'e de Nantes \textbf{1} (1989), 1--14.
\bibitem{ChSh} So-Chin Chen and Mei-Chi Shaw, \textit{Partial differential equations in several complex variables.} Studies in Advanced Mathematics, Vol. 19,  Amer. Math. Soc., 2001.
\bibitem{Ch} M. Christ, \textit{On the $\ovprt $ equation in weighted $L^2$ norms in $\mathbb C^1.$}  J. of Geometric Analysis \textbf{1} (1991), 193--230. 
\bibitem{ChF} M. Christ and S. Fu, \textit{Compactness in the $\ovprt$-Neumann problem, magnetic Schr\"odinger operators, and the Aharonov-Bohm effect.} Adv. Math. \textbf{197} (2005), 1--40.
\bibitem{F} G.B. Folland, \textit{Introduction to partial differential equations.} Princeton University Press, Princeton, 1995.
\bibitem{FS3} S. Fu and E.J. Straube, \textit{Semi-classical analysis of Schr\"odinger operators
and compactness in the $\ovprt $ Neumann problem.} J. Math. Anal. Appl. \textbf{271} (2002), 267--282.
\bibitem{Ga} K. Gansberger, \textit{Compactness of the $\dquer$-Neumann operator.} Dissertation, University of Vienna, 2009.
\bibitem{GaHa} K. Gansberger and F. Haslinger, \textit{Compactness estimates for the $\ovprt $- Neumann problem in weighted $L^2$ - spaces.} Complex Analysis (P. Ebenfelt, N. Hungerb\"uhler, J.J. Kohn, N. Mok, E.J. Straube, eds.), Trends in Mathematics, Birkh\"auser (2010), 159--174.
\bibitem{Has5} F. Haslinger, \textit{Compactness for the $\ovprt $- Neumann problem- a functional analysis approach.} Collectanea Math. (to appear), arXiv: 0912.4406.
\bibitem{HaHe} F. Haslinger and B. Helffer, \textit{Compactness of the solution operator to $\ovprt $ in weighted $L^ 2$ - spaces.} J. of Functional Analysis \textbf{255} (2008), 13--24.
\bibitem{H} L. H\"ormander, \textit{An introduction to complex analysis in several variables.} North Holland, 1990.
\bibitem{Jo} J. Johnsen, \textit{On the spectral properties of Witten Laplacians, their range
  projections and Brascamp-Lieb's inequality.}  Integral Equations Operator Theory \textbf{36} (2000), 288--324.
\bibitem{KM} J.-M. Kneib and F. Mignot, \textit{Equation de Schmoluchowski g\'en\'eralis\'ee.}  
Ann. Math. Pura Appl. (IV) \textbf{167} (1994), 257--298.
\bibitem{McN2} J.D. McNeal, \textit{$L^2$ estimates on twisted Cauchy-Riemann complexes.}
150 years of mathematics at Washington University in St. Louis. Sesquicentennial of mathematics at Washington University, St. Louis, MO, USA, October 3--5, 2003. Providence, RI: American Mathematical Society (AMS). Contemporary Mathematics \textbf{395} (2006), 83--103.
\bibitem{Schn06} G. Schneider, \textit{Non-compactness of the solution operator to
    {$\overline\partial$} on the {F}ock-space in several dimensions.} Math. Nachr. \textbf{278} (2005), 312--317.
\bibitem{Str} E. Straube, \textit{The $L^2$-Sobolev theory of the $\ovprt $-Neumann problem.}     
 ESI Lectures in Mathematics and Physics, EMS, 2010.
\end{thebibliography}
\end{document}